\documentclass[11pt,a4paper]{article}
\usepackage{fullpage}
\usepackage{theorem,graphicx,amssymb,amsmath,url}

\theorembodyfont{\slshape}

\newtheorem{thm}{Theorem}
\newtheorem{lem}[thm]{Lemma}
\newtheorem{cor}[thm]{Corollary}
\newtheorem{prop}[thm]{Proposition}
\newtheorem{defn}[thm]{Definition}
\newtheorem{rmk}[thm]{Remark}
\newtheorem{obs}[thm]{Observation}
\newtheorem{example}[thm]{Example}

\def\QED{\ensuremath{{\square}}}
\def\markatright#1{\leavevmode\unskip\nobreak\quad\hspace*{\fill}{#1}}
\newenvironment{proof}
 {\begin{trivlist}\item[\hskip\labelsep{\bf Proof.}]}
 {\markatright{\QED}\end{trivlist}}

\usepackage{xcolor}
\usepackage[inline]{enumitem}
\usepackage{bbm}
\usepackage{hyperref} 

\DeclareMathOperator{\TL}{TL}

\begin{document}

\title{Transformed flips in triangulations and matchings}

\author{
	Oswin Aichholzer\thanks{Institute for Software Technology,
	Graz University of Technology, Graz, Austria, \newline  {\tt [oaich|bvogt]@ist.tugraz.at} } 
\and
    Lukas Andritsch\thanks{Mathematics and Scientific Computing, University of Graz, Graz, Austria, \newline  {\tt [baurk|lukas.andritsch]@uni-graz.at} }
\and
	Karin Baur\textsuperscript{$\dagger$}
\and
	Birgit Vogtenhuber\textsuperscript{$\ast$}
}

\maketitle

\begin{abstract}
Plane perfect matchings of $2n$ points in convex position are in bijection with triangulations of convex polygons of size $n+2$. Edge flips are a classic operation to perform local changes both structures have in common. In this work, we use the explicit bijection from \cite{AABV2017} to determine the effect of an edge flip on the one side of the bijection to the other side, that is, we show how the two different types of edge flips are related. Moreover, we give an algebraic interpretation of the flip graph of triangulations in terms of elements of the corresponding Temperley-Lieb algebra.
\end{abstract}

\section{Introduction}\label{sec:introduction}

Triangulations and plane perfect matchings are among the most fundamental types of graphs and have a huge variety of applications in different fields of mathematics and computer science.
A direct bijection between plane perfect matchings on $2n$ vertices in convex position and triangulations on $n\!+\!2$ points in convex position is presented in~\cite{AABV2017}.
We start by recalling some definitions and results from there.

We depict perfect matchings with two parallel rows of $n$ vertices each, labeled
$v_1$ to $v_{n}$ and $v_{n+1}$ to $v_{2n}$ in clockwise order, and with non-crossing edges; see \figurename~\ref{fig:basic_drawing}(left). 
To describe triangulations of convex $(n\!+\!2)$-gons we draw $n+2$ points in convex position, labeled $p_1$ to $p_{n+2}$ in clockwise order; see 
\figurename~\ref{fig:basic_drawing}(right). 
For the sake of distinguishability, throughout this paper we will refer to $p_1, \ldots, p_{n+2}$ as \emph{points} and to $v_1, \ldots, v_{2n}$ as \emph{vertices}.

\begin{figure}[htb]
	\centering\includegraphics[scale=0.75, page=1]{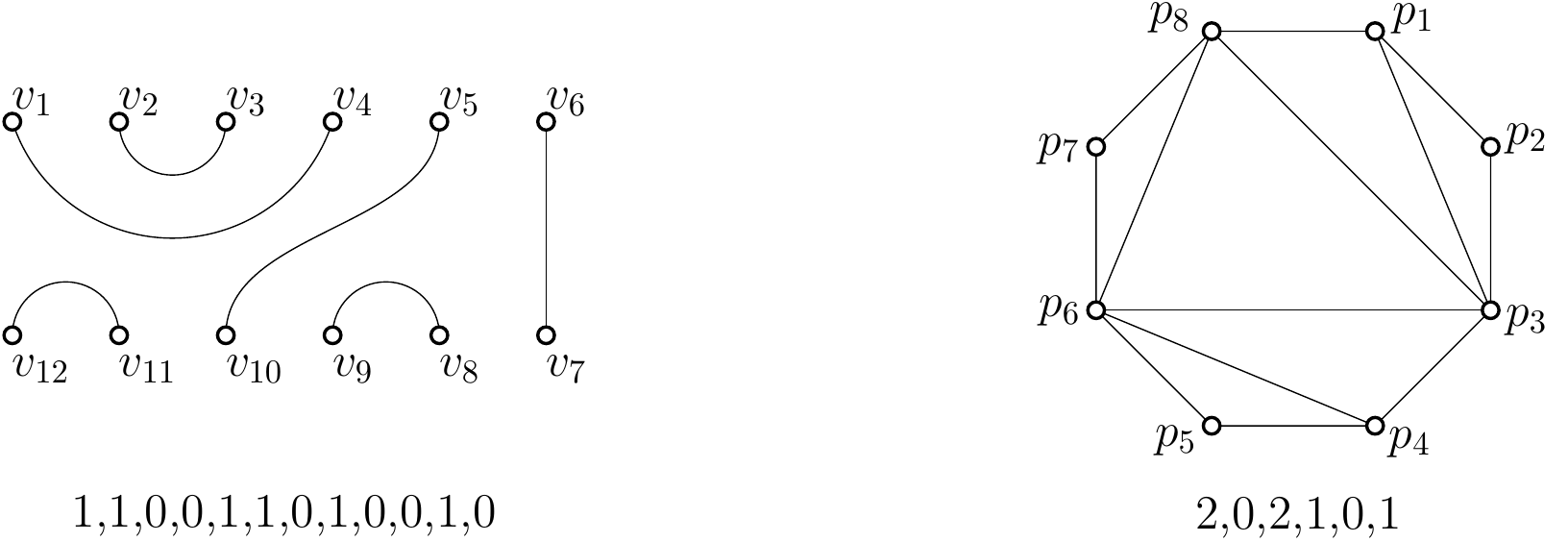}\\
	\caption{A perfect matching (left) and the corresponding triangulation for $n=6$ (right), together with their outdegree sequences.}
	\label{fig:basic_drawing}
\end{figure}

\noindent 
The above defined structures are undirected graphs. 
We next equip their edges (diagonals) with directions: 
an edge $v_iv_j$ ($p_ip_j$) is directed from $v_i$ to $v_j$ ($p_i$ to $p_j$) for $i<j$, that is, each edge is directed from the vertex / point with lower index to the vertex / point with higher index.
This also defines the outdegree of every vertex / point, which we denote as $b_i$ for each vertex $v_i$ and as $d_i$ for each point $p_i$. 
For technical reasons, 
we do not take into account the edges of the convex hull of a triangulation when computing the outdegree of a point $p_i$, with the exception of the edge $p_1p_{n+2}$.
For matchings, the outdegree sequence is a $0/1$-sequence with $2n$ digits, where $n$ digits are $1$ and $n$ digits are $0$. 
For triangulations, first note that the outdegrees of $p_{n+1}$ and $p_{n+2}$ are 0.
Thus we do not lose information when restricting the outdegree sequence of a triangulation to $(d_1,\ldots, d_n)$.
Recall that we do not consider the edges of the convex hull, except for $p_1p_{n+2}$, and thus the number of edges which contribute to the outdegree sequence is exactly $n$. 
We call the sequence $B(M)=B:=(b_1,\ldots, b_{2n})$ of the outdegrees of a matching $M$ (and the sequence $D(T)=D:=(d_1,\ldots, d_n)$ of the first $n$ outdegrees of a triangulation $T$) its \emph{outdegree sequence}; see again \figurename~\ref{fig:basic_drawing}. 
For both structures, this sequence is sufficient to encode the graph, as described in~\cite{AABV2017}.
The two types of outdegree sequences are sufficient to define a bijection between these two types of graphs:
For a given outdegree sequence $B=(b_1, \ldots, b_{2n})$ of a perfect matching,
for $i>0$ the outdegree $d_i$ for the corresponding point of the  triangulation is the number of 1s between the $(i-1)$-st $0$ and the $i$-th $0$ in~$B$. The degree $d_1$ is the number of 1s before the first 0 in~$B$.

\begin{figure}[htb]
	\centering\includegraphics[scale=0.75, page=2]{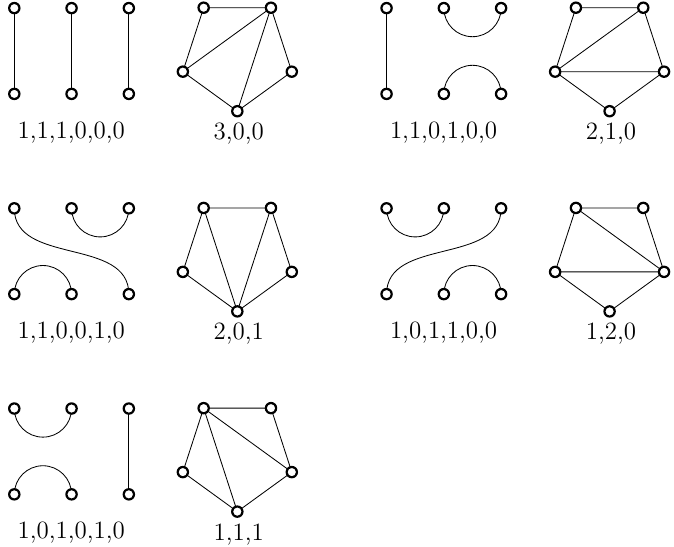}
	\caption{All perfect matchings, triangulations, and outdegree sequences for $n=3$.} 
	\label{fig:n5}
\end{figure}

Given an outdegree sequence $(d_1,\ldots,d_n)$, we obtain $(b_1,\ldots,b_n)$ as follows: for every entry $d_i$, we set the next $d_i$ consecutive elements (possibly none) of $B$ to $1$, then insert a $0$ and continue with $d_{i+1}$.
For a matching $M$, we denote the corresponding triangulation under the described bijection by $T_M$ and for a given triangulation $T$, the corresponding matching is denoted by $M_T$ throughout the article.

Temperley and Lieb introduced in \cite{TL1971} an algebra arising from a special kind of lattice models, which are a key ingredient in statistical mechanics.   
The Temperley-Lieb algebra $\TL_n(\alpha)$ is abstractly defined over a field $k$ by the generators $u_1, \ldots, u_{n-1}$ together with the identity $I$ and an element $\alpha \in k\setminus \{0\}$ obeying the relations

\begin{align}
u_i^2&=\alpha u_i, \ \ 1 \leq i \leq n-1\\
u_i u_j &= u_j u_i, \ \ |i-j| > 1, \  1\leq i,j\leq n-1 \\
u_i u_{i+1} u_i &= u_i,\ \ 1 \leq i \leq n-2 \\
u_{i+1}u_i u_{i+1}&=u_{i+1}, \ \ 1 \leq i \leq n-2.
\end{align} 

The basis of the algebra consists of all reduced words. For example, a basis of $\TL_3(\alpha)$ over the field $k$ is $\{I, u_1, u_2, u_1 u_2, u_2 u_1 \}$, independently of the element $\alpha$.
Kauffman introduced a pictorial representation of the Temperley-Lieb algebras in \cite{K1987}. Each generator corresponds to a plane perfect matching with $n$ vertices on top and bottom labelled $v_1, \ldots, v_n$ and $v_{n+1}, \ldots, v_{2n}$ in clockwise order.
The identity consists of $n$ (straight) propagating lines, the generator $u_i$, for $i=1,\ldots,n-1$, consists of $n-2$ (straight) propagating lines and two arcs between the pairs $(v_i,v_{i+1})$ and $(v_{2n-i},v_{2n-i+1})$ respectively, see \figurename~\ref{fig:algebraic_background_generators}.

\begin{figure}[htb]
	\centering\includegraphics[scale=0.75, page=1]{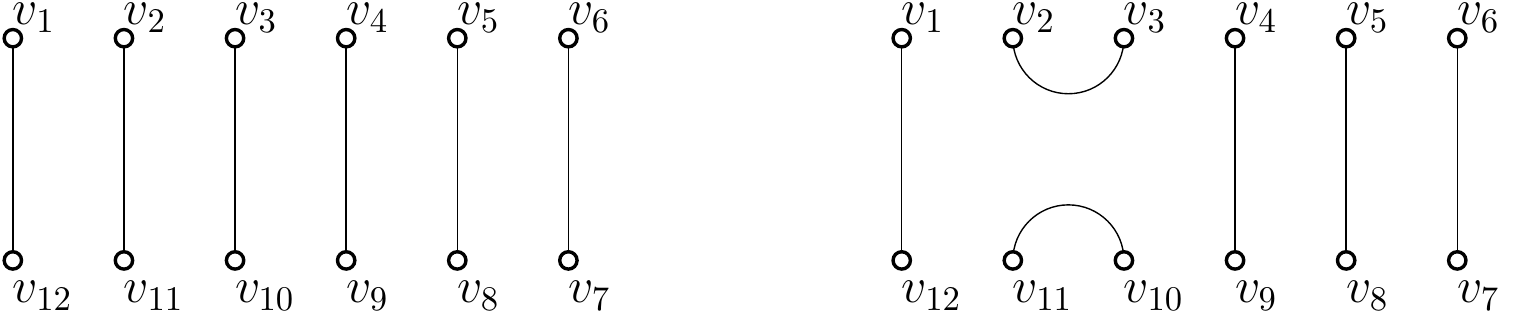}\\
	\caption{The identity $I$ (left) and one of the generators, $u_2$ (right), of $\TL_6(\alpha)$.}
	\label{fig:algebraic_background_generators}
\end{figure}

Products of generators of the algebra are obtained by concatenation of the corresponding matchings from top to bottom. Any loop arising from this is removed and replaced by a factor $\alpha$, e.g. $u_i u_i= \alpha u_i$, see \figurename~\ref{fig:algebraic_background_loops}.

\begin{figure}[htb]
	\centering\includegraphics[scale=0.75, page=3]{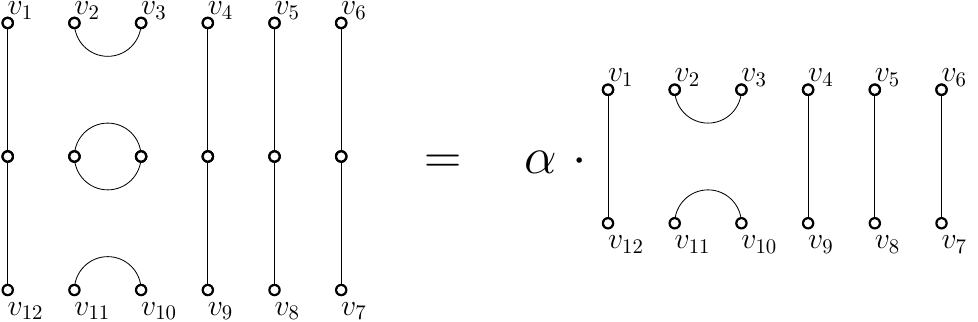}\\
	\caption{Loops are replaced by multiplication with the field element $\alpha$, here: $u_2^2=\alpha u_2$.}
	\label{fig:algebraic_background_loops}
\end{figure}

One can check that all the relations $(1)$-$(4)$ are satisfied.
It is a well known result that the dimension of $\TL_n(\alpha)$ is equal to $C_n=\tfrac{1}{n+1}\binom{2n}{n}$, the $n$-th Catalan number (see \cite{BJ1997} for an example). 
In this article, we consider the interplay between diagrams for $\TL_n(\alpha)$ per se and the generators and will from now on fix $\alpha=1$. We write $TL_n$ for $TL_n(1)$.

A classic operation triangulations and perfect matchings have in common are edge flips, see~\cite{BH2009} for a nice survey. For triangulations this means that a diagonal in a convex quadrilateral is replaced be the other diagonal. For perfect matchings two matching edges which induce a crossing-free 4-cycle are replaced be the other two edges of this cycle. Our goal is to determine the effect of a flip on one side of the bijection on the other side. 

The paper is organized as follows. First we characterize diagonal-flips of triangulations in terms of matchings in Section~\ref{sec:triangulation}. Then we give an algebraic interpretation of the flip graph of triangulations of an $(n\!+\!2)$-gon in terms of elements of the Temperley-Lieb algebra $TL_n$ in Section~\ref{sec:algebra}.  There we focus on the generators of $TL_n$, determining flip distances between them. In Section~\ref{sec:matching}, we define and focus on matching-flips and describe the change of the corresponding triangulations, and derive connections between these two types of flips.

\section{Flips in triangulations and their interpretation for matchings}
\label{sec:triangulation}

In this section we consider a basic transformation operation for triangulations, the so-called diagonal-flip~\cite{BH2009}.
While for triangulations this is a constant size operation, the impact on the corresponding matching is more involved.
In the following we show in detail how a diagonal-flip alters the matching.

\begin{defn}
	Let $T$ be a triangulation of a convex $(n\!+\!2)$-gon containing the
(convex) quadrilateral $p_ip_jp_kp_l$ as a sub graph, $1 \leq i < j <
k < l \leq n\!+\!2$.
A \textit{diagonal-flip} in $T$, is the
exchange of one of the two possible diagonals within this quadrilateral by the other.
\end{defn}
A diagonal-flip that exchagnes the diagonal $p_jp_l$ by the diagonal $p_ip_k$
decreases the outdegree $d_j$ by $1$ and increases the outdegree $d_i$ by $1$, as $1\leq i<j<k<l$; see
Figure~\ref{fig:triangulation_flip_1} for an example. 

\begin{figure}[htb]
	\centering \includegraphics[scale=0.75, page=1]{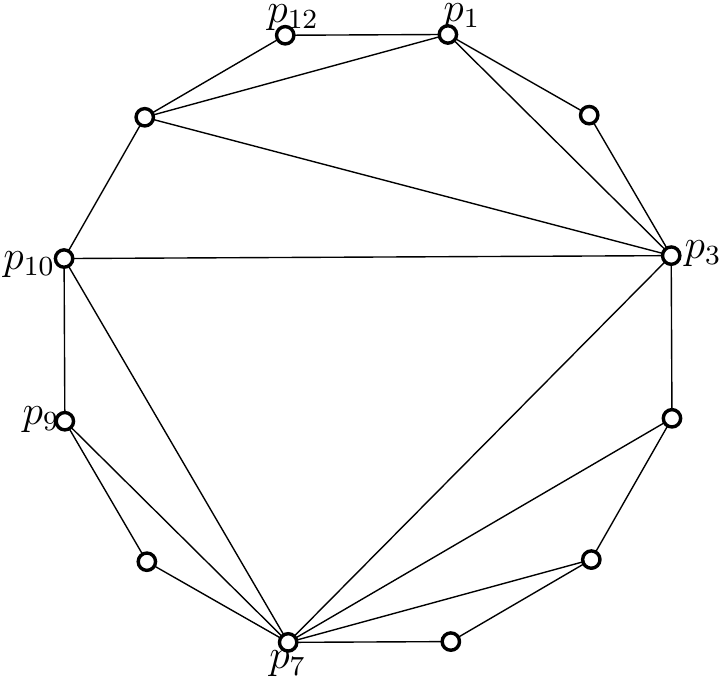}
	\hspace{15 ex}
	\centering \includegraphics[scale=0.75, page=2]{triangulation_flip}
	\caption{A flip in the quadrilateral $p_3p_7p_9p_{10}$ from $p_7p_{10}$ to $p_3p_9$ decreases $d_7$ and increases $d_3$.}
	\label{fig:triangulation_flip_1}
\end{figure} 

The change of the corresponding matching is illustrated in Figure~\ref{fig:triangulation_flip_2}.
Considering the outdegree sequences of these matchings, we see that an entry $1$  of the outdegree sequence is moved from, say, $v_q$ to a vertex $v_p$ with $q>p$. Moreover, all degrees between these two vertices are shifted by one position towards larger indices, see below for details and a more formal description.

\begin{figure}[htb]
	\centering \includegraphics[scale=0.75, page=3]{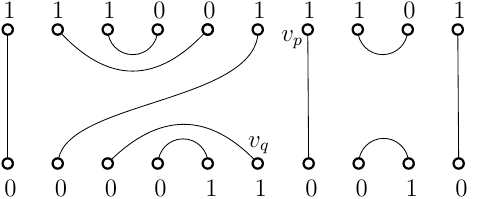}
	\hspace{10 ex}
	\centering \includegraphics[scale=0.75, page=4]{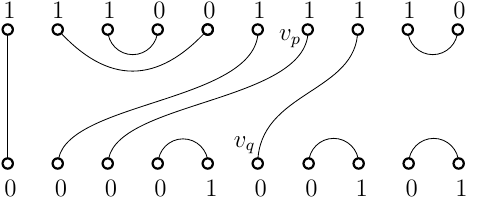}
	\caption{The matchings which correspond to the triangulations of Figure \ref{fig:triangulation_flip_1}.}
	\label{fig:triangulation_flip_2}
\end{figure} 

In order to prove a general characterization of diagonal-flips of a triangulation of a polygon in terms of matchings, we introduce a more convenient presentation.

\begin{defn}
The \textit{single row presentation} of a perfect matching is obtained by drawing the $2n$ vertices
$v_1,\ldots,v_{2n}$ in one row and connecting
the pairs of the matching with non-crossing edges below these vertices.
\end{defn}

Note that the single row presentations are $(2n,n)$-link states as in the language of \cite{RS-A2014}.
To be able to distinguish between the two presentations of a perfect matching, we call edges of the single row presentation \textit{arcs}. 

Using the single row presentation, we can illustrate the effect of a
diagonal flip in terms of matchings; see
Figure \ref{fig:triangulation_flip_3} which corresponds to the diagonal-flip shown in Figure \ref{fig:triangulation_flip_1}.  The arc starting at $v_q$ (which is the right hand neighbor of the endpoint of the arc starting at $v_p$) is
removed, the part between $v_p$ (including) and $v_q$ (excluding), which consists of as many
$1$s as $0$s, is moved one step to the right. The removed arc is
reinserted starting at $v_p$ and ending at the same vertex as before. So the flip corresponds to an extension of the bold arc in Figure \ref{fig:triangulation_flip_3} over the subsequence drawn by dashed lines.

\begin{figure}[htb]
	\centering \includegraphics[scale=0.75, page=5]{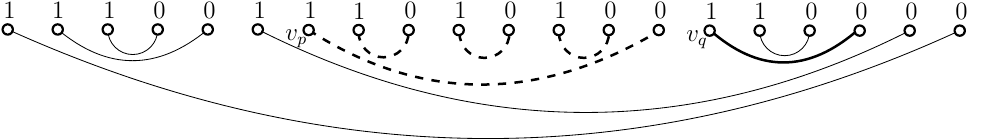}\\
	\vspace{3 ex}
 	\centering \includegraphics[scale=0.75, page=6]{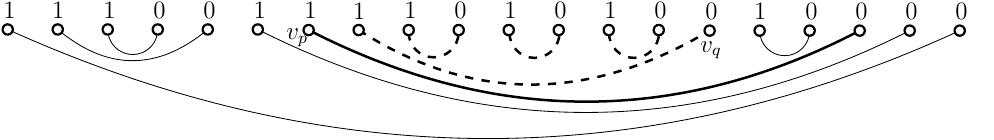}
	\caption{The single row presentations of the matchings of Figure \ref{fig:triangulation_flip_2}.}
	\label{fig:triangulation_flip_3}
\end{figure}  

To describe the effect of a diagonal-flip in general, we need the following definitions.

\begin{defn}In a single row presentation a \textit{section} is a part of an outdegree sequence starting with a $1$ and ending when the numbers of $1$s and $0$s (when going to the right) coincide for the first time. We will also call the collection of arcs corresponding to this outdegree sequence a section.
\end{defn} 

Every section is determined by the arc corresponding to the starting 1 and ending 0.

\begin{defn}
	Every vertex $v_p$ is the start or end vertex of an arc and hence determines a section through this arc. We will write $_{v_{p}}\mathcal{A}$ if $v_p$ is the starting vertex of section $\mathcal{A}$ and $\mathcal{A}_{v_p}$ if it is the ending vertex of a section $\mathcal{A}$.
\end{defn}

For example, each of the short dashed arcs in Figure \ref{fig:triangulation_flip_3} is a section, and the collection of the four dashed arcs forms also a section.

\begin{rmk} \label{rmk:section}
Let $\mathcal{A}$ be a section of the single row presentation of a matching $M$. In terms of the associated triangulation $T_M$, $\mathcal{A}$ corresponds to a union of a triangulated subpolygon with a diagonal (or the edge $p_1p_{n+2}$). This can be seen using the bijection from Section~\ref{sec:introduction}. 
For example, the section $_{v_p}\mathcal{A}$ in the first diagram of Figure \ref{fig:triangulation_flip_3} corresponds to the triangulated subpolygon on vertices $p_3,\ldots,p_7$ and the diagonal $p_3p_{10}$ in the first triangulation of Figure \ref{fig:triangulation_flip_1}.
\end{rmk}

\begin{defn}
	Consider a matching $M$ in its single row presentation and let $v_p$ be the starting vertex of some arc $\alpha$. \\
	Assume that $p>2$ and that $v_{p-1}$ is the ending vertex of a section $\mathcal{A}=\mathcal{A}_{v_{p-1}}$. Then we can combine $\mathcal{A}_{v_{p-1}}$ and the section $_{v_p}\mathcal{A}$ of $\alpha$ as in Figure \ref{fig:triangulation_flip_4}. We call this {\em a merge of $\alpha$ and $\mathcal{A}$}. Note that $\alpha$ already determines $\mathcal{A}$, so we also call the operation {\em a merge for $\alpha$}.\\
	If for $p>1$, $v_{p-1}$ is the starting vertex of some arc $\beta$, we can separate the section of $\alpha$ from the section of $\beta$ as shown in Figure \ref{fig:triangulation_flip_5}. We call this operation the {\em split of the section of $\alpha$ from the section of $\beta$}. Note that $\alpha$ already determines $\beta$. So we may call this operation the {\em split of $\alpha$}.
\end{defn}

\begin{figure}[htb]  
	\vspace{-1ex}
	\centering \includegraphics[scale=0.75, page=7]{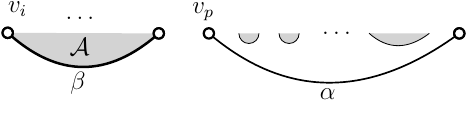}
	\hspace{5 ex} $\mapsto$ \hspace{5 ex}
	\centering \includegraphics[scale=0.75, page=8]{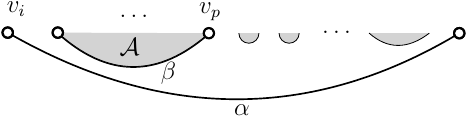}
	\caption{Merging $\alpha$ and $\mathcal{A}v_{p-1}$. The section $\mathcal{A}$ determined by arc $\beta$ is moved one vertex to the right, $\alpha$ now starts at $v_i$.}
	\label{fig:triangulation_flip_4}
\end{figure}  
\begin{figure}[htb]
	\vspace{-1ex}
	\centering \includegraphics[scale=0.75, page=9]{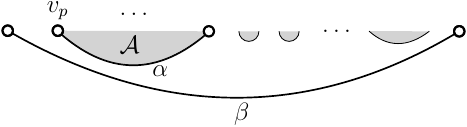}
	\hspace{5 ex} $\mapsto$ \hspace{5 ex}
	\centering \includegraphics[scale=0.75, page=10]{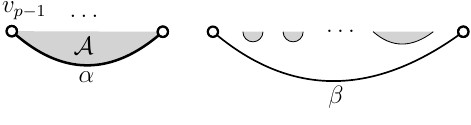}
	\caption{Splitting the section defined by $\alpha$ from the section defined by $\beta$. }
	\label{fig:triangulation_flip_5}
\end{figure}

\begin{rmk}\label{rmk:split_merge}
The following properties are immediately implied by the definition of the split and merge operation.
\begin{itemize}
	\item[i)] Split and merge are inverse operations.
	\item[ii)] Every arc $\alpha=v_pv_q$ for $q>p>1$ determines exactly one split (split of $\alpha$) or one merge (merge of $\alpha$ and $\mathcal{A}$).
	\item[iii)] A merge of $\alpha$ and $\mathcal{A}$ moves section $\mathcal{A}$ by one index to the right so
	that it ends where the arc $\alpha$ started before the merge, and $\alpha$ starts
	now where $\mathcal{A}$ originally started; See Figure~\ref{fig:triangulation_flip_4}.
\end{itemize}
\end{rmk}

\begin{lem}\label{lem:flip_to_merge_and_split}
Every diagonal-flip in a triangulation $T$ corresponds to a merge or to a split in the single row presentation of $T$. 
\end{lem}
\begin{proof}
Every diagonal-flip in a triangulation $T$ decreases and increases the degrees $d_i$ and $d_j$ by $1$ while all other outdegrees remain unchanged. Using the transformation from the outdegree sequence of a triangulation, $D(T)$, to the outdegree sequence of the related matching, $B(M_T)$, as in Section \ref{sec:introduction}, it follows that a $1$ (the $1$ which comes from the edge $p_ip_l$), say for $v_p$, is moved to a vertex with higher index (say $v_q$, just before the vertex with degree $1$ corresponding to the first outgoing edge of $p_j$ or the $0$ in case we had $d_j=0$ before the flip). 

All outdegrees of the matching with index smaller than $p$ or larger than $q$ are not affected, and the outdegrees $b_{p+1}$ to $b_q$ have to be shifted by 1 to the left. So the diagonal flip from $p_ip_k$ to $p_jp_l$ is precisely the split operation described above. Note that also the number and order of $0$s and $1$s which are shifted form a section, as the edge $p_ip_j$ splits of a sub triangulation of $T$ which together with the edge $p_ip_k$ forms the shifted section (as stated in Remark \ref{rmk:section}). 

The reverse flip $p_jp_l$ to $p_ip_k$ consequently corresponds to the merge operation.
\end{proof}

\begin{lem}\label{lem:merge_and_split_to_flip}
Every split and every merge in a single row presentation of a matching $M$ corresponds to a flip in the triangulation $T_M$.
\end{lem}
\begin{proof}\label{lem:split_to_flip}
All arguments of the proof of Lemma~\ref{lem:flip_to_merge_and_split} work in both directions.
\end{proof}

The preceeding two lemmata imply the following corollary.

\begin{cor}\label{cor:triangulation_flip}
	Every flip of a diagonal of a given triangulation of the $(n\!+\!2)$-gon corresponds to either a merge or split of one of the arcs $\alpha$ in the matching, where the index of the starting vertex of $\alpha$ is greater than one. The reverse also holds. 
\end{cor}	

\begin{rmk}
	As the single row presentation consists of $n$ arcs, and $n-1$ of them can perform either a merge or a flip, we obtain a well known result: every triangulation of an $(n\!+\!2)$-gon has $n\!-\!1$ neighbors in the diagonal-flip graph.
\end{rmk}

\begin{rmk}
It was brought to the authors' attention, that the result of Corollary \ref{cor:triangulation_flip} are related to the theory of Tamari lattices and Dyck paths. The former was introduced in \cite{Ta1962}; see \cite{Mu2012} for a recent survey. However, the known operations on Dyck paths are different from the operations merge and split defined above. We will study this connection in upcoming work.  
\end{rmk}

In Section~\ref{sec:matching} we will consider the reverse, namely the impact of a flip in a matching (exchanging a pair of matching edges) on the related triangulation. Before that, we will use the obtained insight on the relation between diagonal-flips and changes in a matching in the next section to get bounds on the flip distance between generators for  the Temperley-Lieb algebra~$TL_n$.

\section{Towards  an algebraic interpretation of the triangulation flip graph}
\label{sec:algebra}
For this section, let $n>0$ be fixed.

\subsection{Generators as triangulations}\label{sec:triangulation_for_generators}

The generators $I$ and $u_i$, $1\leq i \leq n\!-\!1$ of the Temperley-Lieb algebra $TL_n$ correspond to particular triangulations of the $(n\!+\!2)$-gon following the bijection from \cite{AABV2017}.

\begin{prop}
	The identity $I$ corresponds to the fan triangulation at point $p_1$ (every diagonal is incident with $p_1$), and the $u_i$ consist of the following diagonals:
	\begin{enumerate}
		\item[(1)] The edge $e_i=p_{n-i+1}p_{n-i+3}$
		\item[(2)] The edge $f_i=p_2p_{n-i+3}$
		\item[(3)] A fan of cardinality $i\!-\!1$ at $p_1$, consisting of the diagonals $p_1p_j$, $n-i+3 \leq j \leq n+1$.
		\item[(4)] A fan of cardinality $n\!-\!i\!-2$ at $p_2$, consisting of the diagonals $p_2p_j$, $4\leq j \leq n-i+1$.
	\end{enumerate} 
	For $i=1$ and $i=n-2$, the fan at $p_1$ and $p_2$, respectively, consists of $0$ edges and for $i=n\!-\!1$, the edges $e_{n-1}$ and $f_{n-1}$ coincide.
\end{prop}
\begin{proof}
	The correspondence follows immediately by using the bijection from \cite{AABV2017}: The outdegree sequence of the matching corresponding to the generator $u_i$ is symmetric and consists of $i$ times 1, followed by a single 0, $(n-i-1)$ ones, then $(n-i-1)$ times 0, followed by a single 1 and $i$ times 0. Beside the first 1, which corresponds to the edge $p_1p_{n+2}$ of the boundary of the convex hull of the triangulation, each 1 corresponds to a diagonal. The last 1 determines the diagonal $e_i$, the $(n-i-1)$ ones correspond to diagonals involving $p_2$ (fan and diagonal $f_i$), and the other $(i-1)$ 1s correspond to  diagonals at point $p_1$. 
\end{proof}

Figure \ref{fig:triangulation_of_generators_4} shows the generator $u_4$ of $TL_{10}$ as an example. The diagonal $p_7p_9=e_4$, a fingerprint of the triangulation, in the sense that all the other diagonals of a generator are determined immediately (using the list above) once the diagonal $e_i$ is drawn. The remaining diagonals are $p_2p_9=f_4$ and the two fans at $p_1$ and $p_2$ consisting of $2$ and $3$ diagonals, respectively.

\begin{figure}[htb]
	\centering\includegraphics[scale=0.75, page=4]{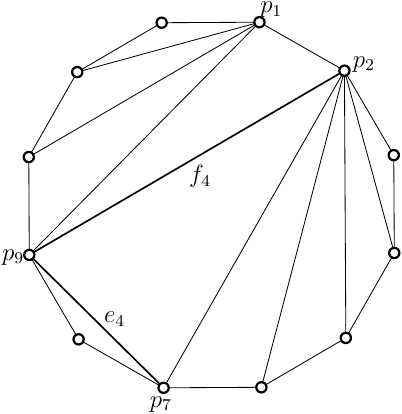}
	\caption{The triangulation of the generator $u_4$ of $TL_{10}$.}
	\label{fig:triangulation_of_generators_4}
\end{figure}

\subsection{Flip distance between generators}
\label{sec:flip_distance_of_generators}

In this section we consider the flip distance $d(u_i,u_j)$ between each pair $u_i$, $u_j$ of generators of the Temperley-Lieb algebra $TL_{n}$ in terms of diagonal flips in the corresponding triangulations. To show that the distances are optimal we will use the following lemma on crossings between diagonals of the convex $n$-gon. Note that two diagonals $p_ip_j$ and $p_kp_l$ with $i<j$ and $k<l$ cross if and only if either $i<k<j<l$ or $k<i<l<j$.

\begin{lem}
\label{lem:crossing}
Let $T_1$ and $T_2$ be two triangulations of a convex $n$-gon. Let $G$ be the graph which contains all the diagonals of $T_1$ and $T_2$ and let $k$ be the number of diagonals of $T_1$ which are crossed by at least one diagonal of $T_2$ in $G$. Then the flip distance between $T_1$ and $T_2$ is at least~$k$.
\end{lem}
\begin{proof}
Every diagonal from $T_1$ which is crossed in $G$ needs to be removed in the flip sequence before we can obtain $T_2$. As every flip changes only one diagonal this requires at least $k$ flips.
\end{proof}

Note that this bound will usually not be tight and that in general $k$ will be much smaller than the number of crossings. Moreover the role of $T_1$ and $T_2$ are symmetric in the sense that, despite the fact that in $G$ one diagonal might be crossed by several other diagonals,  the number of diagonals in $T_1$ which are crossed by diagonals in $T_2$ is the same as the number of diagonals in $T_2$ which are crossed by diagonals in $T_1$. 

\begin{prop}\label{prop:flip_distance_of_generators}
The minimal number 
of flips between the triangulations of two generators~is
	\begin{align*}
	d(u_i,u_{i+1}) &=\left\{ \begin{array}{ll} 3 & \text{ if} \ i \neq n-2 \\
	2  & \text{ if} \ i =n-2 \end{array} \right. \\
	d(u_i,u_{i+k}) &=\left\{ \begin{array}{ll} k+1 & \text{ if} \ i \neq n-k-1 \\
	k  & \text{ if} \ i =n-k-1 \end{array} \right. \\
	d(u_i,u_{i-1}) &=\left\{ \begin{array}{ll} 3 & \text{ if} \ i \neq n-1 \\
	2  & \text{ if} \ i =n-1 \end{array} \right. \\
	d(u_i,u_{i-k}) &=\left\{ \begin{array}{ll} k+1 & \text{ if} \ i \neq n-1 \\
	k  & \text{ if} \ i =n-1 \end{array} \right.,
	\end{align*}
	where $1 \leq i \leq n-1$ ,  $k \geq 2$ and all indices are positive integers.
\end{prop}
\begin{proof} 
First observe that any flip sequence in a triangulation can be reversed by just performing the flips in inverted order, i.e. $d(u_i,u_j) = d(u_j,u_i)$ for all $i\neq j$. Thus we only have to consider the first two cases. 

The key idea to obtain the exact numbers of flips between triangulations of generators is to consider the "special" diagonal $e_i$, which, as mentioned above, defines the triangulation of a generator. Furthermore, by decreasing and increasing the fans of $p_1$ and $p_2$ accordingly, the majority of flips will already be done. To argue that our bounds are best possible we will use the statement of Lemma~\ref{lem:crossing}. Similar as in this lemma we say that two diagonals cross if they cross in the joint graph of the two triangulations of the two generators.

 We consider the two cases separately.
\begin{description}
\itemsep0pt
\item[$d(u_i,u_{i+1})$]  If $i=n-2$, we have to flip $e_{n-2}$ and then $f_{n-2}$. If $i\neq n-2$, we flip $e_{i}$, then $f_{i}$ and finally the diagonal $p_2p_{n-i+3}$ to the diagonal $p_{n-i+2}p_{n-i+4}=e_{i+1}$. In both cases the number of flips are optimal by Lemma~\ref{lem:crossing} as at the beginning there exist 2 (respectively 3) diagonals which are crossed by edges of the final triangulation.
\item[$d(u_i,u_{i+k})$] If $i \neq n-k-1$ then first flip $f_i$ followed by the flip of $e_i$, both increasing the fan at $p_1$. Then flip $k-2$ diagonals of the fan of $p_2$, in the order $p_2p_{n-i+1},p_2p_{n-i},\ldots p_2p_{n-i-k+4}$. Finally flip $p_{2}p_{n-i+k+2}$ to $p_{n-i+k+1}p_{n-i+k+3}=e_{i+k}$. This is a sequence of $k+1$ flips. To see that this is optimal consider all edges of the start triangulation which are crossed. These are $e_i$ and $f_i$, the $k-2$ diagonals which are flipped from the fan at $p_2$ to the fan at $p_1$ and the edge $p_{2}p_{n-i+k+2}$ which gets flipped to $e_{i+k}$, see Figure \ref{fig:flip_sequence_generators_minimality}.
In total this gives $k+1$ edges, and thus Lemma~\ref{lem:crossing} implies that the sequence is minimal.

 If $i=n-k-1$, then the last flip from $p_{2}p_{n-i+k+2}$ to $e_{i+k}$ is not necessary, as $f_{n-1}=e_{n-1}$. Thus the length of the flip sequence is $k$, and optimality can be argued as before. 
\end{description}
	\vspace{-4.8ex}
\end{proof}
	
\begin{figure}[htb]  
		\centering \includegraphics[scale=0.75, page=3]{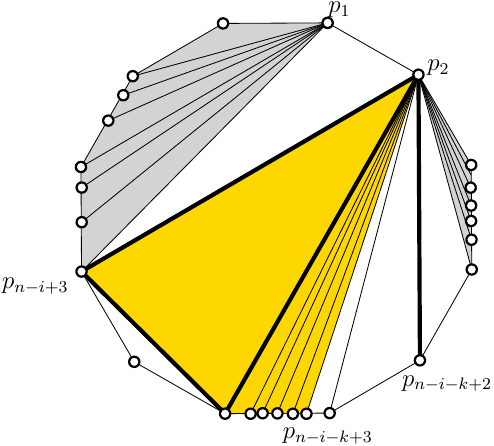}
		\hspace{5 ex} \hspace{3 ex}
		\centering \includegraphics[scale=0.75, page=4]{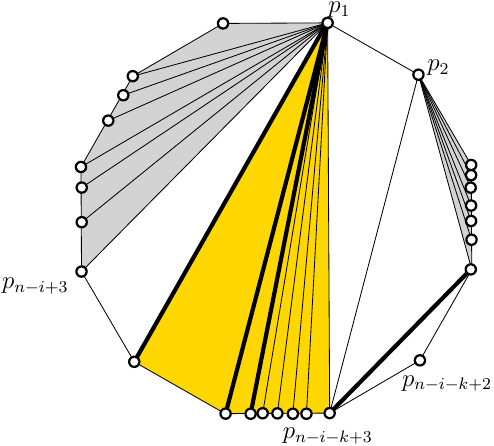}
		\caption{Triangulations of $u_i$ (left) and $u_{i+k}$ (right).}
		\label{fig:flip_sequence_generators_minimality}
	\end{figure}

\subsection{Flip neighbors of generators}
In this section, we give a full list of the flip neighbors of the generators $I$, $u_1$, $\ldots$, $u_{n-1}$  of the Temperley-Lieb algebra $TL_n$ in the flip graph of the $n\!+\!2$-gon. 

\begin{lem}\label{lem:flip_neighbors_identity}
For $k\in \{1,\ldots,n-1\}, u_{n-1} \cdots u_{n-k}$ corresponds to the triangulation which consists of the diagonal  $p_{k+1}p_{k+3}$ and $n\!-\!2$ diagonals incident with $p_1$.
\end{lem}
\begin{proof}
For $k\in \{1,\ldots,n-1\}$, the multiplication of the product of generators $u_{n-1}\ldots u_{n-k+1}$ with the generator $u_{n-k}$ from right effects the shift of the arc at the bottom part of $u_{n-1}\ldots u_{n-k+1}$ to the left by one vertex in the corresponding matching, as shown in Figure \ref{fig:generator_flip_1}. 
	Hence the matchings corresponding to the products $u_{n-1} \cdots u_{n-k}$ for $k=1,\ldots, n-1$ consist of $n-2$ propagating lines, the arc $v_{n-1}v_n$ and the arc $v_{n+k}v_{n+k+1}$. By using the bijection described in Section \ref{sec:introduction}, the arc $v_{n+k}v_{n+k+1}$ corresponds to the diagonal $p_{k+1}p_{k+3}$ in the $(n\!+\!2)$-gon and the further 1s of the outdegree sequence of the matching correspond to the remaining $n\!-\!2$ diagonals, all incident with~$p_1$.  
\end{proof}

\begin{cor}
	The flip neighbors of the identity $I\in \TL_n$ are the $n\!-\!1$ elements \[\{u_{n-1}u_{n-2} \cdots u_{n-k} \mid k=1,\ldots,n-1\}.\]
\end{cor}
\begin{proof}
The identity $I$ corresponds to the fan triangulation at $p_1$.	
	For fixed \mbox{$k\!\in\!\{1,\ldots,n\!-\!1\}$}, flipping the diagonal $p_1p_{k+2}$ of this fan triangulation to the diagonal $p_{k+1}p_{k+3}$ gives the triangulation corresponding to the element $u_{n-1}u_{n-2} \cdots u_{n-k}$ as shown in Lemma~\ref{lem:flip_neighbors_identity}. 
\end{proof}

\begin{figure}[htb]
	\centering \includegraphics[scale=0.75, page=13]{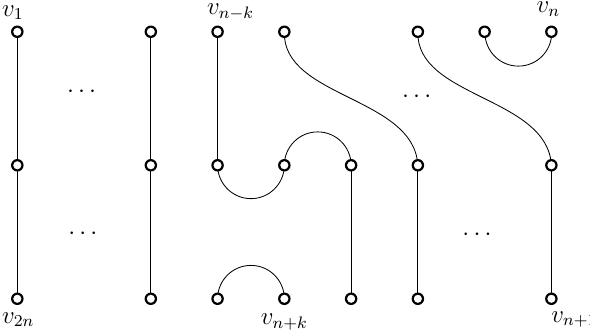}
	\caption{Pictorial presentation of the product of $u_{n-1}\cdots u_{n-k+1}$ and $u_{n-k}$ in $TL_n$.}
	\label{fig:generator_flip_1}
\end{figure} 

Next we give a complete description of all flip neighbors of the remaining generators of $TL_{n} $.
In preparation for Proposition \ref{prop:generator_flip_1}, we provide an example illustrating the connection between flip neighbors for generators of $\TL_n$ and for generators of $\TL_{n+1}$. 

\begin{figure}[h]
	\centering \includegraphics[scale=0.75, page=14]{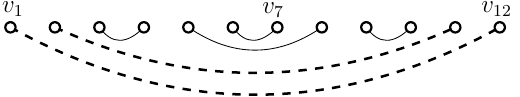}
	\hspace{7 ex}
	\centering \includegraphics[scale=0.75, page=15]{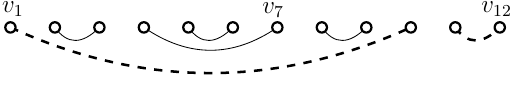}
	\caption{Single row presentations of $u_3$ and $u_2u_1u_5u_4$ in $TL_8$.}
	\label{fig:generator_flip_2}
\end{figure} 
\begin{figure}[h]
	\vspace{-2ex}
	\centering \includegraphics[scale=0.75, page=16]{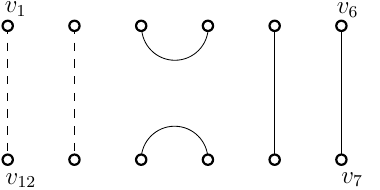}
	\hspace{12 ex}
	\centering \includegraphics[scale=0.75, page=17]{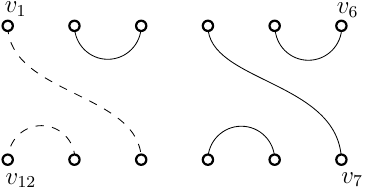}
	\caption{Pictorial presentation of $u_3$ (left) and $u_2u_1u_5u_4$ (right) in $TL_8$.}
	\label{fig:generator_flip_3}
\end{figure} 
\begin{figure}[h]
	\vspace{-2ex}
	\centering \includegraphics[scale=0.75, page=18]{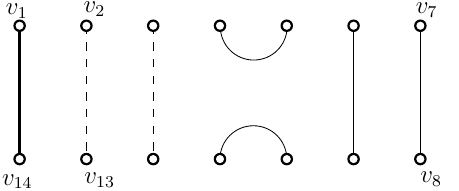}
	\hspace{7 ex}
	\centering \includegraphics[scale=0.75, page=19]{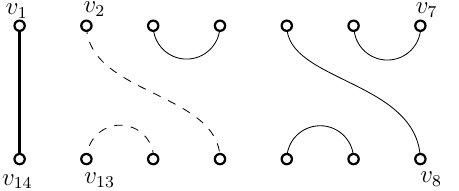}
	\caption{Pictorial presentation of $u_4$ (left) and $u_3u_2u_6u_5$ (right) in $TL_9$.}
	\label{fig:generator_flip_4}
\end{figure} 
\begin{figure}[h]
	\vspace{-2ex}
	\centering \includegraphics[scale=0.75, page=20]{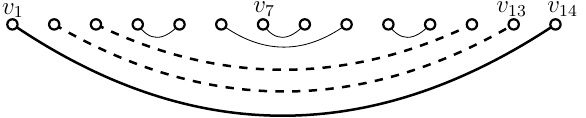}
	\hfill 
	\centering \includegraphics[scale=0.75, page=21]{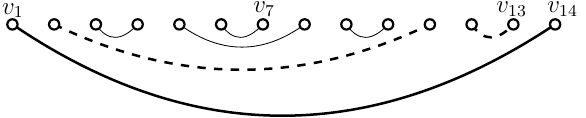}
	\caption{Single row presentation of $u_4$ (left) and $u_3u_2u_6u_5$ (right) in $TL_9$.}
	\label{fig:generator_flip_5}
\end{figure} 

\begin{example}\label{ex:inductive_flip}
	Figure \ref{fig:generator_flip_2} shows a flip of the triangulation corresponding to $u_3\in TL_8$ in the single row presentation, where the arcs $v_1v_{12}$ and $v_2v_{11}$ of $u_3$ are split. The element on the right of Figure \ref{fig:generator_flip_2} corresponds to $U:=u_2u_1u_5u_4=\prod_{k\in J}u_k$, with $J=\{2,1,5,4\}$ as can be seen by looking at the matching on the right of Figure \ref{fig:generator_flip_3}. When we just add a propagating line on the left of both matchings, we get the elements $u_{4}$ and $U^{+1}:=u_3u_2u_6u_5=\prod_{k\in J^{+}}u_k$, with $J^{+}=\{3,2,6,5\}$ in $TL_9$, as shown in Figure \ref{fig:generator_flip_4}. That these two elements are joined by a flip (the split of $v_2v_{13}$ and $v_3v_{12}$) can easily be seen by their single row presentations in Figure \ref{fig:generator_flip_5}.
\end{example}

Using the idea of Example \ref{ex:inductive_flip}, we can prove the following inductive formula determining $n-2$ of $n-1$ flip neighbors of the generators $u_2,\ldots, u_{n} \in TL_{n+1}$, if the flip neighbors of the generators $u_1,\ldots, u_{n-1} \in TL_{n}$ are known.

\begin{prop}\label{prop:generator_flip_1}
	Let $J$ be an ordered index set of $\{1,\ldots, n-1\}$ and $J^{+1}$ the set where every index $k\in J$ is increased by $1$. \\
	Let $U=\prod_{k\in J}{u_k}$ be a flip neighbor of $u_i$ in $TL_n$ for some index set $J$ and $1\leq i \leq n-1$. Then $U^{+1}:=\prod_{k\in J^{+1}}{u_k}$ is a flip neighbor of $u_{i+1}$ in $TL_{n+1}$.
\end{prop}
\begin{proof}
	For $n=2$, consider $J=\{\}$ and the flip neighbor $\prod_{k\in J}{u_k}=I\in TL_{2}$ of $u_1\in TL_{2}$. As $J=\{\}=J^{+1}$ and $\prod_{k\in J^{+1}}{u_k}=I\in TL_{3}$ is a flip neighbor of $u_2\in TL_{3}$, the statement holds.    
	For given $n$, let $u_i \in T_{n}$ for some $1\leq i\leq n-1$ and $u_{i+1} \in TL_{n+1}$ be generators of the two algebras. The corresponding matching of $u_{i+1}$ differs from the matching of $u_i$ by an additional propagating line $v_1v_{2n+2}$ and a shift of all indices of the vertices of the matching of $u_i$ by 1. By the characterization of a flip as a split or a merge of two neighbored arcs in the single row presentation, the statement follows immediately, as a shift of indices does not change the neighborhood of the arcs. 
\end{proof}

\begin{obs}
	If a matching in $TL_{n}$ contains the edge $v_{1}v_{2n}$, then the corresponding triangulation of the $(n\!+\!2)$-gon contains the diagonal $p_1p_{n+1}$.
\end{obs}

This observation is exactly the situation covered by Proposition \ref{prop:generator_flip_1} above. The $n\!-\!2$ flips of an arbitrary triangulation $T'$ of a $(n\!+\!1)$-gon are also performable flips for the triangulation $T$ of the $(n\!+\!2)$-gon, which contains all diagonals of $T'$ and the additional triangle $p_1p_{n+1}p_{n+2}$. 

\begin{obs}\label{obs:product_of_elements}
Reversing the argument of Lemma~\ref{lem:flip_neighbors_identity}, we easily obtain that multiplication of an element $u_{1}\cdots u_{n-k-1}$ with $u_{n-k}$ shifts the arc $v_{n+k+1}v_{n+k+2}$ of the matching corresponding to $u_{1}\cdots u_{n-k-1}$ by one in the other direction. Hence the matching $u_{1}\cdots u_{n-k-1} u_{n-k}$ contains the arcs $v_1v_2$ and $v_{n+k}v_{n+k+1}$, and $n\!-\!2$ propagating lines.
\end{obs}

In order to prove the following lemma, which gives the product form of the remaining flip neighbor for the elements $u_2,\ldots, u_{n} \in TL_{n+1}$, we first recall the description of three types of products of generators.

\begin{rmk}\label{rmk:remaining_flip_generators}
As Observation \ref{obs:product_of_elements} shows, the matching corresponding to the element $u_1u_2$  consists of the arcs $v_{1}v_{2}$ and $v_{2n-2}v_{2n-1}$, the propagating line $v_3v_{2n}$ and $n-3$ straight propagating lines $v_{k}v_{2n-k+1}$ for $3\leq k \leq n$.\\
By the same idea as Lemma~\ref{lem:flip_neighbors_identity}, we obtain that the matching corresponding to $u_{n-2}u_{n-3}\cdots u_1$ consists of the arcs $v_{n-2}v_{n-1}$ and $v_{2n-1}v_{2n}$, the straight propagating line $v_{n}v_{n+1}$ and the $n\!-\!3$ propagating lines $v_{k}v_{2n-1-k}$ for $1\leq k \leq n-3$.\\
Multiplying for $3\leq k \leq n-1$ the element $u_{k-1}\cdots u_1$ with the element $u_{n-1}\cdots u_{k+1}$, we obtain the matching shown in Figure \ref{fig:generator_flip_6}. Here the dotted line emphasizes the well known fact, that for $1\leq i,j \leq n\!-\!1$, the product $u_iu_j$ of two generators of $TL_n$ commute, that is, $u_iu_j=u_ju_i$, if and only if $|i-j|\geq 2$. Hence these two elements commute and we can write  $u_{n-1}\cdots u_{k+1} u_{i-k}\cdots u_1$ for the product.
\end{rmk}	

\begin{figure}[htb]
	\centering \includegraphics[scale=0.75, page=22]{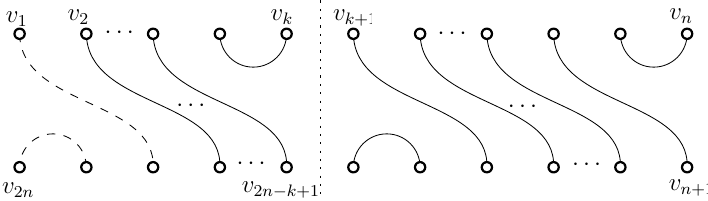}
	\caption{Result of multiplication of $u_{i-k}\cdots u_1$ and $u_{n-1}\cdots u_{k+1}$.}
	\label{fig:generator_flip_6}
\end{figure}      

\begin{lem}\label{lem:flip_neighbor}
	Let $2\leq i \leq n-1$. Then the generator $u_i\in TL_{n}$ is a neighbor of the element
	\begin{align*}
	&u_1 u_2 \ \ &\text{ if } &i=2, \\
	&u_{n-1}\cdots u_{i+1}u_{i-1}\cdots u_1 \ &\text{ if } 3\leq \ &i \leq n-2, \\
	&u_{n-2}\cdots u_1 \ &\text{ if } &i=n-1
	\end{align*}	
	in the flip graph.
\end{lem}

\begin{proof}
The split of the arcs $v_{2}v_{2n-1}$ and $v_1v_{2n}$ for $i\neq 2$ and the split of $v_1v_{2n}$ and $v_2v_3$ for $i=2$ respectively yields to a single row presentation, which corresponds to the matching (described in Remark \ref{rmk:remaining_flip_generators}) of the elements stated in the lemma in every case.
\end{proof}

For $n=8$ and $i=3$, Figure \ref{fig:generator_flip_1} and Figure \ref{fig:generator_flip_2} show the flip neighbor of $u_3$ in $TL_8$, which is described by Lemma~\ref{lem:flip_neighbor}. All the other flip neighbors of $u_3$ can be constructed by Proposition~\ref{prop:generator_flip_1}.\\
It only remains to determine the flip neighbors of $u_1$ in $TL_n$. For this, we first recall the description of three particular products of generators.

\begin{lem}
	The flip neighbors of $u_1 \in TL_{n}$ are 
	\begin{align*}
	&u_2u_1\\
	&u_1u_2\cdots u_{n-1} \\
	&u_{n-1} \cdots u_k u_1 \text{ for } 3\leq k \leq n-1.
	\end{align*}
\end{lem}
\begin{proof}
Consider the single row presentation of $u_1$.
The first two elements arise from merging $v_3v_{2n-2}$ with $v_1v_2$ and with $v_{2n-1}v_{2n}$, respectively. The other $n\!-\!3$ elements arise by the splitting of the arcs $v_{4}v_{2n-3}$, $v_{5}v_{2n-4}$, $\ldots$, $v_{n-1}v_{n+2}$ and $v_{5}v_{2n-4}$, $v_{6}v_{2n-5}$,$\ldots$, $v_{n}v_{n+1}$, respectively. We obtain the matchings corresponding to the stated elements of $TL_n$ analogously as in Remark~\ref{rmk:remaining_flip_generators}, and get the desired result.
\end{proof}

\begin{cor}
The pairs $\{u_ku_{k+1},u_{k+1}\}$ and $\{u_{k+1}u_k,u_k\}$ are flip neighbors for $1\leq k \leq n\!-\!2$ in $TL_n$.
\end{cor}

Note that this corollary gives the flip sequence $u_{n-1} - u_{n-1}u_{n-2} - u_{n-2}$ in $TL_n$ and as the two generators are not flip neighbors, we obtain that $d(u_{n-2},u_{n-1})=2$, a result already stated in Proposition \ref{prop:flip_distance_of_generators}. 

The flip neighbors of the generators $I, u_1,\ldots u_4$ of $TL_5$ are listed in the table below:
\begin{center}
	\begin{tabular}{l|l}
		generator & neighbors \\
	     $I$   & $u_4$, $u_4u_3$, $u_4u_3u_2$, $u_4u_3u_2u_1$ \\
	     $u_1$ & $u_1u_2u_3u_4$, $u_2u_1$, $u_4u_3u_1$, $u_4u_1$ \\
		 $u_2$ & $u_1u_2$, $u_3u_2$, $u_2u_3u_4$, $u_4u_2$\\
		 $u_3$ & $u_4u_2u_1$, $u_3u_2$, $u_4u_3$, $u_3u_4$ \\
		 $u_4$ & $u_3u_2u_1$, $u_3u_2$, $u_3u_4$, $I$   
	\end{tabular} 
\end{center}

\section{Flips in matchings and their interpretation for triangulations}
\label{sec:matching}

In Section~\ref{sec:triangulation} we considered the diagonal-flip in
a triangulation and its impact on the related matching. In this
section we take a closer look at the reverse direction: What does a general
flip of two edges in a matching change for the related triangulations?
Although this question has no direct consequence for the underlying
algebraic structures, a better understanding of this relation might
help to fully understand the importance of triangulations for
generators.\\

We consider the visualization of a matching as vertices $v_1,\ldots,v_{2n}$ lying on a circle cyclically.
The area bounded by this circle is partitioned by the (non-crossing) edges of the matching into several inner faces.

\begin{defn}\label{defn:matching_flip}
Consider four indices $1 \leq i < j < k < l \leq 2n$ which are incident to a common inner face of the partition induced by a  plane perfect matching.
A flip in this matching is the change between the two possible plane matchings of these indices, that is, between the edge-pair $v_iv_l$, $v_jv_k$ and the edge-pair $v_iv_j$, $v_kv_l$, respectively.
We call such an operation a \textit{matching-flip}. 
\end{defn}

\begin{figure}[thb]
	\centering\includegraphics[scale=0.75, page=5]{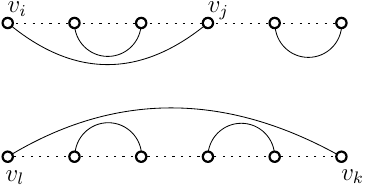} 
		\hspace{15 ex}
	\includegraphics[scale=0.75, page=6]{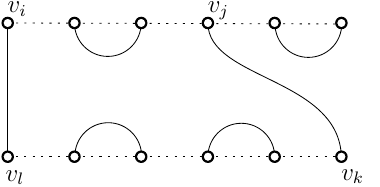}
	\caption{Flip of a matching. Replace edges $v_iv_l$ and $v_jv_k$ (left) by $v_iv_j$ and $v_kv_l$ (right). The dotted lines indicate parts of the bounding circle.}
	\label{fig:matching_flip_1}
\end{figure}

\begin{rmk}
	\hfill \\ \vspace{-4ex}
\begin{itemize}
	\item[(i)] 
		The two possibilities in Definition~\ref{defn:matching_flip} are the only plane matchings for these indices and that in both cases the four indices share a common inner face. Thus, if a pair of edges can be flipped, the resulting matching is uniquely defined, and a flip can always be reversed by a second matching-flip.
\item[(ii)] 
	In terms of multiplication of elements of $TL_n$, a matching-flip is not always reversible, as shown in Figure~\ref{fig:matching_flip_1}: If propagating lines vanish under a flip, it is only perform-able in one direction by multiplication with elements of the TL-algebra.
For example, in Figure~\ref{fig:matching_flip_1} this is only possible from left to right.
\item[(iii)] A subset of matching-flips, where $j=i+1\neq n+1$ (or $l=k+1$) can be described easily by a multiplication of the corresponding element $X$ of the Temperley-Lieb algebra by the element $u_i$ from left for $i<n$ ($k<n$)  or right for $i>n$ ($k>n$), respectively.
\end{itemize}
\end{rmk}

We now 
describe matching-flips in terms of triangulations.
Considering the general case of a matching-flip we assume that between
any two vertices there are at least two other vertices. 
That is, $j > i+2$, $k > j+2$, and $l > k+2$. Other cases, where at least $3$ indices are neighbored, have
been considered in Section~\ref{sec:triangulation}, where only a single edge of each triangulation is flipped. The remaining cases will be discussed
at the end of this section.
Recall that the edges of a matching (the diagonals of a triangulation
of a polygon) are directed from the vertex (point) with lower index to
the vertex (point) with higher index and as described in \ref{sec:introduction}
$b_i$ ($d_i$) denotes the outdegree of the vertex $v_i$
(point $p_i$) for $1 \leq i \leq 2n$ ($1 \leq i \leq n+2$). Furthermore, the outdegree sequence of a matching $M$ is denoted by $B(M)$, the outdegree sequence of a triangulation $T$ is denoted by $D(T)$. \\
Let $B(M)=(b_1,\ldots,b_{2n})$ be the outdegree sequence of a matching $M$ with flippable edges $v_iv_l$ and $v_jv_k$. Then

\vspace{-5ex}
\begin{align*}
&  i \hspace{1.9cm}  j \hspace{1.9cm}  k \hspace{1.85cm}  l \\
B(M)= (1, \ldots, \ &\mathbf{1}, \ 1, \ldots, 0, \ \mathbf{1}, \ 1, \ldots, 0, \ \mathbf{0}, \ 1, \ldots, 0, \ \mathbf{0}, \ldots, 0), 
\end{align*}
where $i$,$j$,$k$ and $l$ are at the position of the entries $b_i$,$b_j$,$b_k$ and $b_l$ respectively.  As edges $v_iv_l$ and $v_jv_k$ are flippable, $v_{i+1}v_{j-1}$, $v_{j+1}v_{k-1}$ and $v_{k+1}v_{l-1}$ are forced to form sections in the single row presentation of $M$ and are in particular edges of $M$. Therefore $b_{i+1}=b_{j+1}=b_{k+1}=1$ and $b_{j-1}=b_{k-1}=b_{l-1}=0$ in $B(M)$. 
We define $D(T_M)=:(a_1,\ldots,a_n)$ as the outdegree sequence of the triangulation obtained by $M$. Then there are indices $p$ and $q$ with $1\leq p<q\leq n$, where the bijection described in Section \ref{sec:introduction} defines $a_p$ of $D(T_M)$ as sum of 1s starting with $b_j=b_{j+1}=1$ till the next 0 in $B(M)$ and $a_{q+1}$ as the sum of 1s in $B(M)$ starting with $b_{k+1}=1$. Note that by this definition of $q$, $a_q$, is the number of 1s between $b_{k-1}=0$ and $b_k=0$ and hence $a_q=0$. Moreover, $a_{q-1}\leq 1$,  as at least one of $b_{k-2}$ or $b_{k-3}$ is $0$,
otherwise one of the edges emanating from $v_{k-2}$ or $v_{k-3}$
would have to cross the edge $v_jv_k$, a contradiction to the fact that the edges  $v_iv_l$ and $v_jv_k$ are flippable. Furthermore, $a_p\geq 2$ and $a_{q-1}\geq 1$ and the values $a_p,\ldots,a_{q-1}$, $a_{q+1}$ are crucial for the matching-flip, as stated in following proposition.  

\begin{prop}
Let $B(M)$ be the outdegree sequence of a matching $M$ with flippable edges $v_iv_l$ and $v_jv_k$ and $B(M')$ be the outdegree sequence of the matching $M'$ which consists of the flipped edges $v_iv_j$ and $v_kv_l$ and all further edges coincide with those of $M$. Assume that the outdegree sequence of the triangulation $T_M$ is of the form

\begin{align*}
D(T_M)=(a_1,\ldots,a_{p-1},a_{p},a_{p+1},a_{p+2},\ldots,a_{q-1},0,a_{q+1},a_{q+2},\ldots,a_n)
\end{align*}  for indices $1\leq p < q \leq n$ and non-negative integers $a_1,\ldots,a_{q-1},a_{q+1},\ldots a_n$ as defined above. Then $D(T_{M'})$ is of the form
\begin{center}
	\begin{tabular}{|c|l|c|c|c|c|c|l|c|c|c|c|l|c|}
		$d_1$ & $\ldots$ & $d_{p-1}$ & $d_p$ & $d_{p+1}$ & $d_{p+2}$ & $d_{p+3}$ & $\ldots$ & $d_{q-1}$ & $d_{q}$ & $d_{q+1}$ & $d_{q+2}$& $\ldots$ & $d_n$ \\
		\hline
		$a_1$ & $\ldots$& $a_{p-1}$ & $0$ & $a_p-1$ & $a_{p+1}$ & $a_{p+2}$ & $\ldots$ & $\ldots$ & $a_{q-1}$ & $a_{q+1}+1$ & $a_{q+2}$& $\ldots$ & $a_n$ \\
	\end{tabular}\ .
\end{center}
\end{prop}

\begin{proof}
The outdegree sequences $B(M)$ and $B(M')$ are of the form

\vspace{-4ex}
\begin{align*}
B(M)= (1, \ldots, \ &\mathbf{1}, \ 1, \ldots, 0, \ \mathbf{1}, \ 1, \ldots, 0, \ \mathbf{0}, \ 1, \ldots, 0, \ \mathbf{0}, \ldots, 0) \\
&  i \hspace{1.9cm}  j \hspace{1.9cm}  k \hspace{1.85cm}  l \\
B(M')= (1, \ldots, \ &\mathbf{1}, \ 1, \ldots, 0, \ \mathbf{0}, \ 1, \ldots, 0, \ \mathbf{1}, \ 1, \ldots, 0, \ \mathbf{0}, \ldots, 0).
\end{align*} 
We consider the corresponding outdegree sequence $D(T_M)$ for the triangulation $T_M$ which is related to the matching $M$,
\begin{center}
	\begin{tabular}{|c|l|c|c|c|c|l|c|c|c|c|l|c|}
		 $d_1$ &$\ldots$ & $d_p$ & $d_{p+1}$ & $d_{p+2}$ & $d_{p+3}$ & $\ldots$ & $d_{q-1}$ & $d_{q}$ & $d_{q+1}$ & $d_{q+2}$ & $\ldots$ & $d_{n}$ \\
		\hline
		 $a_1$ & $\ldots$ & $a_p$ & $a_{p+1}$ & $a_{p+2}$ & $a_{p+3}$ & $\ldots$ & $a_{q-1}$ & $0$ & $a_{q+1}$ & $d_{q+2}$ & $\ldots$ & $a_{n}$ \\ 
	\end{tabular}
\end{center}
with indices $p$ and $q$ and non-negative integers $a_1,\ldots,a_{q-1},a_{q+1},\ldots a_n$ as stated above.

The change of the outdegree of $v_j$ from 1 to 0 means that the outgoing edges of $p_p$ are reduced by one and shifted to $p_{p+1}$. Consequently the outgoing edges of $p_{p+1}$ are shifted to $p_{p+2}$ and so on till the point $p_q$. The point $p_q$ had no outgoing edges, which corresponds to the fact that the vertices $v_k$ and also $v_{k-1}$ had outdegree 0. The change of the outdegree of $v_k$ from 0 to 1 means that $p_{q+1}$ gets one additional outgoing edge. All remaining points are not affected. In other words, there is one edge for which the vertex it emanates from changes (from $p_p$ to $p_{q+1}$), and a whole block of the triangulation that is shifted by one point. Hence the outdegree sequence $D(T_{M'})$ of the triangulation $T_{M'}$ corresponding to the matching-flip $M'$ is of the form
\begin{center}
	\begin{tabular}{|c|l|c|c|c|c|l|c|c|c|c|l|c|}
	  $d_1$ &	$\ldots$ & $d_p$ & $d_{p+1}$ & $d_{p+2}$ & $d_{p+3}$ & $\ldots$ & $d_{q-1}$ & $d_{q}$ & $d_{q+1}$ & $d_{q+2}$ & $\ldots$ & $d_{n}$ \\
		\hline
	$a_1$ &	$\ldots$ & $0$ & $a_p-1$ & $a_{p+1}$ & $a_{p+2}$ & $\ldots$ & $a_{q-2}$ & $a_{q-1}$ & $a_{q+1}+1$ & $d_{q+2}$ & $\ldots$ & $d_{n}$ \\
	\end{tabular} \ ,
\end{center}
as claimed.
\end{proof}

\begin{figure}[htb]  
	\centering \includegraphics[scale=0.7, page=7]{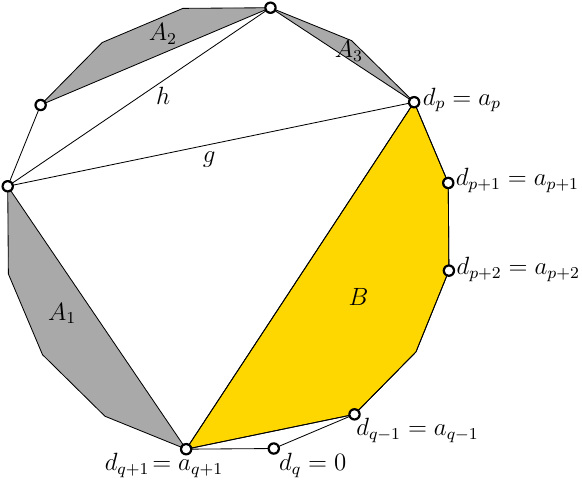}
	\hspace{3 ex}
	\centering \includegraphics[scale=0.7, page=8]{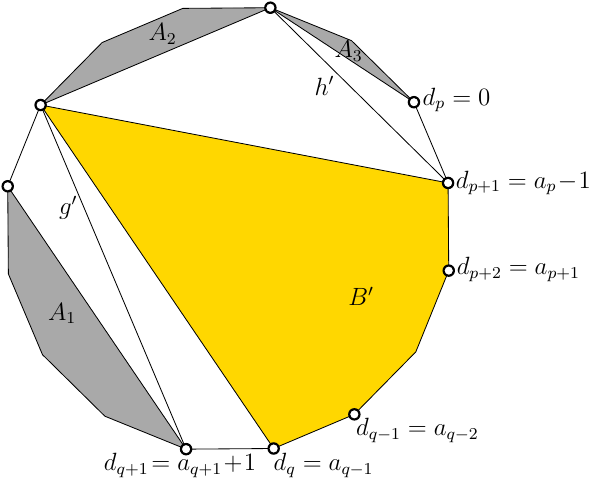}
	\caption{The impact of the matching-flip of Figure~\ref{fig:matching_flip_1} on a triangulation.
                 The (triangulated) parts $A_1$, $A_2$, $A_3$ are not affected, edges $g$ and $h$ are moved, and the triangulated part $B$ is shifted. Labels indicate the outdegrees.}
	\label{fig:matching_flip_2}	
\end{figure}

Figure~\ref{fig:matching_flip_2} illustrates the affect of the matching-flip shown in Figure~\ref{fig:matching_flip_1} to the underlying triangulations.
Diagonal $g$ in $T_M$ is the diagonal which emanates from $p_p$ before the flip, and
$p_{q+1}$ in $T_{M'}$. The part $B$ of the triangulation $T_M$ is rotated
by one point, and parts $A_1$ to $A_3$ (which might be empty sets) are not directly affected.

The diagonals belonging to the triangulation $B$ of the subpolygon
$p_p,p_{p+1},\ldots,p_{q+1}$ rotate to $B'$ in $T_{M'}$ as follows: the start of
the diagonal (lower index) is shifted by one in clockwise direction
(to the vertex with the next higher index). The end of the diagonal is
shifted in the same direction. But if the edge was incident to
$p_{q+1}$ it has to skip the triangulated subpolygon $A_1$, that is, it is
shifted by more than one point.

As mentioned, the start point of $g$ is changed from $p_p$ to
$p_{q+1}$, and its end point is shifted by one. The diagonal $h$
still emanates from the same point, and its endpoint is shifted to $p_{q+1}$.

Before we address the special cases, where some of the indices $i,j,k,l$ are neighbored, note the following connection between the matchings $M$ and $M'$ and their triangulations $T_M$ and~$T_{M'}$. 

\begin{rmk}
The diagonal $g$ in $T_M$ is determined by the edge $v_jv_k$ in $M$ and the diagonal $g'$ in $T_{M'}$ is defined by the edge $v_kv_l$ in $M'$. Furthermore, the diagonal $h$ is determined by the edge starting at $v_{i+1}$. Finally, the triangulated part $B$ of $T_M$ is determined by the section enclosed by the edge $v_{j+1}v_{k-1}$ in $M$.
\end{rmk}

\noindent There are three special cases that may occur (maybe simultaneously):
\begin{enumerate}
\itemsep0pt
\item[(1)] If $j=i+1$, then the diagonal $h$ does not exist, which follows immediately from the remark above.
\item[(2)] If $k=j+1$, then there is no part $B$ in $T_M$, again following from the remark above.
\item[(3)] If $l=k+1$, then the parts $A_1$ and $A_2$ of $T_M$ share a common point. 
\end{enumerate}
When two cases occur simultaneously then this leads to the following results already discussed in Section~\ref{sec:triangulation}: 
If (1) and (2) occur simultaneously, then the matching-flip from $M$ to $M'$ corresponds to a flip of the edge $g$ to $g'$ in the triangulations.
If (2) and (3) occur simultaneously, then the matching-flip from $M$ to $M'$ generates a flip of $h$ in $T_M$  to $h'$ in $T_{M'}$ and a flip from $g$ in $T_M$ to $h$ in $T_{M'}$ as there is no part $B$. Together this implies a flip of the diagonal $g$ in $T_M$ to the diagonal $h'$ in $T_{M'}$.

\section{Conclusion and future work}

In our work we describe diagonal-flips of triangulations in terms of related matchings as well as matching-flips in terms of the corresponding triangulations. This gives a characterization of the connection between those two types of flips. We provide an algebraic interpretation of the flip graph of triangulations of an $(n\!+\!2)$-gon in terms of the elements of the Temperley-Lieb algebra $TL_n$, and also determine the flip distances between its generators. 

There are open questions for a complete understanding of the flip graph in terms of basis elements of the TL-algebra. The main focus in \cite{AABV2017} is the "colored case", the bijection between $k$-colored perfect matchings of $2mk$ vertices, which correspond to the Fuss-Catalan algebras $TL_{mk,k}$ and $k$-gonal tilings. We want to address flips of $k$-colored perfect matchings of $2mk$ vertices, the flip graph of the $k$-gonal tilings of the $(2mk\!+\!2)$-gon and hence obtain the algebraic connection of the flips in terms of the Fuss-Catalan algebras $TL_{mk,k}$. 
Preliminary results in this direction have already been obtained. 

\paragraph{Acknowledgements.} 
Research for this work is supported by the Austrian Science Fund (FWF) grant W1230. K. B. was furthermore supported by FWF grant P30549 and by a Royal Society Wolfson Fellowship.  We thank Cesar Ceballos for bringing Tamari Lattices to our attention.
Moreover, we thank Stefan Wilfinger for implementing a tool to visualize flips
in matchings and triangulations.

\bibliographystyle{abbrv}
\bibliography{bibliography}

\begin{thebibliography}{1}

\bibitem{AABV2017}
O.~Aichholzer, L.~Andritsch, K.~Baur, and B.~Vogtenhuber.
\newblock Perfect $k$-colored matchings and $(k+2)$-gonal tilings.
\newblock {\em Graphs and Combinatorics}, 34(6):1333--1346, 2018.

\bibitem{BJ1997}
D.~Bisch and V.~Jones.
\newblock Algebras associated to intermediate subfactors.
\newblock {\em Invent. Math.}, 128(1):89--157, 1997.

\bibitem{BH2009}
P.~Bose and F.~Hurtado.
\newblock Flips in planar graphs.
\newblock {\em Computational Geometry}, 42(1):60 -- 80, 2009.

\bibitem{K1987}
L.~{Kauffman}.
\newblock {S}tate models and the {J}ones polynomial.
\newblock {\em Topology}, 26(3):395--407, 1987.

\bibitem{Mu2012}
F.~M{\"u}ller-Hoissen, J.~M. Pallo, and J.~Stasheff.
\newblock {\em Associahedra, {T}amari lattices and related structures: {T}amari
  memorial {F}estschrift}, volume 299.
\newblock {S}pringer {S}cience \& {B}usiness {M}edia, 2012.

\bibitem{RS-A2014}
D.~Ridout, Y.~Saint-Aubin, et~al.
\newblock Standard modules, induction and the structure of the
  {T}emperley-{L}ieb algebra.
\newblock {\em Advances in Theoretical and Mathematical Physics},
  18(5):957--1041, 2014.

\bibitem{Ta1962}
D.~Tamari.
\newblock The algebra of bracketings and their enumeration.
\newblock {\em Nieuw Arch. Wisk. (3)}, 10:131--146, 1962.

\bibitem{TL1971}
H.~Temperley and E.~Lieb.
\newblock {R}elations between the {P}ercolation and {C}oloring {P}roblems and
  other {G}raph-{T}heoretical {P}roblems associated with regular {P}lanar
  {L}attices: {S}ome exact results for the {P}ercolation {P}roblem.
\newblock {\em Proc. Roy. Soc.}, 322:147--280, 1997.

\end{thebibliography}

\end{document}